\documentclass[12pt]{article}

\usepackage{graphicx}
\usepackage{subcaption}
\usepackage{bm}
\usepackage[top=1in, bottom=1in, left=1in, right=1in]{geometry}
\usepackage{amsmath}
\usepackage{blkarray}
\usepackage{amsthm}
\usepackage{amsfonts}
\usepackage{array}
\usepackage{listings}
\usepackage{physics}
\usepackage{bbm}
\usepackage{xcolor}
\usepackage{url}
\usepackage{mathtools,hyperref}
\usepackage{xcolor}
\usepackage[utf8]{inputenc}
\usepackage{natbib}

\def\eqd{\buildrel \rm d \over =}

\newcommand{\Var}{\mathrm{Var}}

\DeclareRobustCommand{\prob}[1][P]{\ensuremath {\mathbb{#1}}}
\DeclareRobustCommand{\EX}[2][{\mathbb{E}}]{\ensuremath {#1}\left[ {#2} \right]}

\newtheorem{theorem}{Theorem}
\newtheorem{lemma}{Lemma}
\newtheorem{corollary}{Corollary}

\theoremstyle{definition}

\newtheorem{remark}{Remark}

\title{On a Tail Bound for Root-Finding in Randomly Growing Trees}
\author{Sam Justice\footnote{Department of Statistics and Actuarial Science, University of Iowa; contact email: samuel-justice@uiowa.edu} \ and N. D. Shyamalkumar\footnotemark[1]}
\date{\today}

\begin{document}

\maketitle

\begin{abstract}
We re-examine a lower-tail upper bound for the random variable
$$X=\prod_{i=1}^{\infty}\min\left\{\sum_{k=1}^iE_k,1\right\},$$
where $E_1,E_2,\ldots\stackrel{iid}\sim\text{Exp}(1)$. This bound has found use in root-finding and seed-finding algorithms for randomly growing trees, and was initially proved as a lemma in the context of the uniform attachment tree model. We first show that $X$ has a useful representation as a compound product of uniform random variables that allows us to determine its moments and refine the existing nonasymptotic bound. Next we demonstrate that the lower-tail probability for $X$ can equivalently be written as a probability involving two independent Poisson random variables, an equivalence that yields a novel general result regarding indpendent Poissons and that also enables us to obtain tight asymptotic bounds on the tail probability of interest.
\end{abstract}

\section{Introduction}

Concentration inequalities, which provide bounds on the tail probabilities of random variables, represent a vital and highly popular subject of mathematical research. These inequalities find applications in a variety of fields, including geometry, statistics, empirical process theory, machine learning, and randomized algorithms, to name a few. In this brief note, we reconsider an inequality that has been used in a particular area of the theory of randomly growing trees. A randomly growing tree consists of a sequence of inductively defined trees $\{T_n\}_{n\in\mathbb{N}}$, where $T_n$ is formed from $T_{n-1}$ by introducing a new vertex and attaching it to an existing vertex of $T_{n-1}$ according to some probabilistic rule. The inequality was first proved and utilized in \cite{bubeck2017} in the context of root-finding algorithms for the uniform attachment tree model, and it was subsequently used again in \cite{reddad2019} to prove a result regarding seed-finding in this same model. A variation on the inequality was also proved in \cite{khim2017} to aid in the analysis of root-finding for diffusions over regular trees. Note that the term root-finding algorithm refers to any method that identifies (in the limit) the first vertex (``root'') of a randomly growing tree with some specified probability. Similarly, a seed-finding algorithm seeks to partially or fully recover the first several vertices (``seed'') of a randomly growing tree with high probability. The tail bound appears as Lemma 2 in \cite{bubeck2017} and is stated as follows:

\vspace{2mm}

\noindent\textbf{Lemma 2 from \cite{bubeck2017}}: Let $E_1,E_2,\ldots\stackrel{iid}\sim\text{Exp}(1)$ and let
\begin{equation}\label{X1}
X=\prod_{i=1}^{\infty}\min\left\{\sum_{k=1}^iE_k,1\right\}.
\end{equation}
Then for any $t>0$, we have that
\begin{equation}\label{X2}
\prob(X \leq t)\leq6t^{1/4}.
\end{equation}

\vspace{2mm}

We note that \eqref{X2} is intended to be used for small $t$, as it gives a nontrivial bound only for $t<(1/6)^4\approx0.000772$. Lemma 12 of \cite{khim2017} proved that the same tail bound holds for
\begin{equation}\label{XB}
X=\prod_{i=1}^{\infty}\min\left\{-\sum_{k=1}^i\log(B_k),1\right\},
\end{equation}
where $B_1,B_2,\ldots\stackrel{iid}\sim\text{Beta}(\beta,1)$ for $\beta\in(0,1]$. Since $B_1 \eqd U^{1/\beta}$ for $U \sim \text{U}(0,1)$, we have using the uniform-exponential relationship that $-\log(B_1)\sim\text{Exp}(\beta)$. Thus the case $\beta=1$ corresponds to the definition of $X$ in \eqref{X1}.

The random variable $X$ arises in the context of a particular method for finding the root/seed of a randomly growing tree. Specifically, letting $V(T)$ denote the vertex set of a tree $T$ and $(T,v)_{u\downarrow}$ denote the subtree starting at $u$ in the rooted tree $(T,v)$, consider the function $\varphi_T:V(T)\to\mathbb{N}$ defined by
\begin{equation}\label{phi}
\varphi_T(v)=\prod_{u \in V(T)\setminus\{v\}}|(T,v)_{u\downarrow}|.
\end{equation}

Thus $\varphi_T(v)$, which is referred to in \cite{khim2017} as the subtree product estimator, is the product of all the sizes of the subtrees of the rooted tree $(T,v)$. It can be thought of as a measure of vertex centrality in $T$, with smaller values of $\varphi_T$ corresponding to more central vertices. Given the product form of $\varphi_T$, it is perhaps unsurprising that $X$ comes into play in proofs involving the estimator. Interestingly, the vertex minimizing the function $\varphi_T$ turns out to be the maximum likelihood estimate for the root of a diffusion over a regular tree, and $\varphi_T$ can also be seen as a relaxation of the likelihood functions for the root and seed in the uniform attachment model. It has been proven for these models (see \cite{bubeck2017}, \cite{khim2017}, and \cite{reddad2019}) that picking the $K$ vertices with the smallest values of $\varphi_T$ produces a successful root/seed-finding algorithm for $K$ large enough. The variable $X$ appears in an intermediate step of these proofs that involves upper-bounding the probability that a vertex $v \in V(T)$ becomes more central than the root vertex $r$ in the limit of a sequence of randomly growing trees $\{T_n\}_{n\in\mathbb{N}}$:
$$\limsup_{n\to\infty}\prob(\varphi_{T_n}(v)\leq\varphi_{T_n}(r)).$$

In this paper, we consider a general setting that encompasses both of the preceding definitions of $X$. In particular, we define $X$ as in \eqref{X1}, but now we let $E_1,E_2,\ldots\stackrel{iid}\sim\text{Exp}(\lambda)$ for an arbitrary $\lambda>0$. We carry out a careful examination of $X$ en route to obtaining both nonasymptotic and asymptotic bounds for $\prob(X \leq t)$. To begin, we demonstrate that $X$ has a useful representation as a compound (Poisson) product of iid standard uniform random variables. This representation allows us to easily calculate the moments of $X$, and it also leads to a quick refinement of the $t^{1/4}$ rate shown in \eqref{X2}. After this, we prove that $\prob(X \leq t)$ can be equivalently expressed as a probability involving two independent Poisson random variables. This equivalence immediately yields a general result regarding independent Poissons, and we also use it to obtain tight asymptotic bounds on $\prob(X \leq t)$.

\section{A Useful Representation and Some Consequences}

We begin by proving that $X$ can be written as a Poisson product of iid standard uniform random variables. Our proof uses some fundamental results from the theory of Poisson processes. In particular, we recall that for a Poisson process $\{N(t):t\in[0,\infty)\}$ with rate $\lambda$, the number of arrivals in any time interval of length $t$ has a $\text{Pois}(\lambda t)$ distribution. Also, conditional on the number of arrivals by time $t$, the arrival times of $N(t)$ have the same joint distribution as the order statistics from a $\text{U}(0,t)$ distribution. See, for example, \cite{ross1996} for these and other basic details on Poisson processes.

\begin{lemma}\label{lemma1}
For $\lambda>0$, let $E_1,E_2,\ldots\stackrel{iid}\sim\text{Exp}(\lambda)$ and let
\begin{equation}\label{X3}
X=\prod_{i=1}^{\infty}\min\left\{\sum_{k=1}^iE_k,1\right\}.
\end{equation}
Then we have that
\begin{equation}\label{X4}
X\eqd\prod_{i=1}^NU_i,
\end{equation}
where $N\sim\text{Pois}(\lambda)$, $\{U_i\}_{i\in\mathbb{N}}$ is an iid sequence of $\text{U}(0,1)$ random variables independent of $N$, and the product is understood to equal $1$ if $N=0$.
\end{lemma}

\begin{proof}
Consider a Poisson process with rate $\lambda$ and arrival times given by the sequence
$$\left\{\sum_{k=1}^iE_k\right\}_{i\in\mathbb{N}}.$$

Moreover, let $N$ denote the number of arrivals by time $1$ for this Poisson process, so that $N\sim\text{Pois}(\lambda)$. We clearly have that
$$
X=
\begin{cases}
1 & \text{if} \ N=0,\\
\prod_{i=1}\limits^N\sum_{k=1}^iE_k & \text{if} \ N>0.
\end{cases}
$$

Given that $N=n$, we also see that
$$\left(E_1,E_1+E_2,\ldots,\sum_{k=1}^nE_k\right)\eqd\left(U_{(1)},\ldots,U_{(n)}\right),$$
where $U_{(1)},\ldots,U_{(n)}$ are the order statistics corresponding to a sample of size $n$ from a $\text{U}(0,1)$ distribution. It follows that
$$
X \eqd X^*=
\begin{cases}
1 & \text{if} \ N=0,\\
\prod\limits_{i=1}^NU_{(i)}=\prod\limits_{i=1}^NU_i & \text{if} \ N>0,
\end{cases}
$$
where $\{U_i\}_{i\in\mathbb{N}}$ is an iid sequence of $\text{U}(0,1)$ random variables independent of $N$.
\end{proof}

This compound product representation easily yields the moments of $X$.

\begin{corollary}\label{cor1}
For any $\beta>-1$, we have that $\EX{X^{\beta}}=e^{-\frac{\beta\lambda}{1+\beta}}$. In particular, this implies that $\EX{X}=e^{-\lambda/2}$ and $\Var(X)=e^{-2\lambda/3}-e^{-\lambda}$.
\end{corollary}

The next result is a simple consequence of Corollary \ref{cor1}. It shows that we can obtain a rate of $t^{\alpha}$ for any $\alpha\in(0,1)$ for the lower tail of $X$, and it furthermore gives the optimal rate corresponding to such $\alpha$. Note that this is a refinement of the original result shown in \eqref{X2}, which gives the rate $t^{1/4}$.

\begin{theorem}\label{theorem1}
For any $\alpha\in(0,1)$, we have for all $t>0$ that
\begin{equation}\label{X5}
\prob(X \leq t) \leq e^{\frac{\alpha\lambda}{1-\alpha}}t^{\alpha}.
\end{equation}

In particular, by optimizing the bound with respect to $\alpha$, it follows that
\begin{equation}\label{X6}
\prob(X \leq t) \leq \exp\left\{-\left(\sqrt{-\log t}-\sqrt{\lambda}\right)_+^2\right\},
\end{equation}
where $x_+=\max\{x,0\}$.
\end{theorem}

\begin{proof}
Using Corollary \ref{cor1}, we have for $\alpha\in(0,1)$ and $t>0$ that
\begin{align*}
\prob(X \leq t)&=\prob(X^{-\alpha} \geq t^{-\alpha})\\
&\leq\EX{X^{-\alpha}}t^{\alpha}\qquad\hbox{(Markov's Inequality)}\\
&=e^{\frac{\alpha\lambda}{1-\alpha}}t^{\alpha}.
\end{align*}

Optimizing the preceding expression with respect to $\alpha$, we see that it is minimized at $(1-\sqrt{\lambda/(-\log t)})_+$, which yields \eqref{X6}.
\end{proof}

\begin{remark}\label{remark1}
We note that Theorem \ref{theorem1} provides a better constant than that given in \eqref{X2}. Also, we observe that since $1/X$ does not have a finite moment generating function in a neighborhood of zero, we resort to finding the best moment bound. In this connection it is worth pointing out that the best moment bound is in general always tighter than the Chernoff bound, see \cite{philips1995}. Another interesting aside is that $-\log X$ is a sub-exponential random variable.
\end{remark}

\section{Tight Asymptotic Bounds}

The following result further exploits the characterization of $X$ from Lemma \ref{lemma1} as a compound product of uniforms to show that $\prob(X \leq t)$ has a novel representation as a probability involving two independent Poisson random variables. While the result is intriguing in its own right, we will also see later that it provides the key to obtaining tight asymptotic bounds for the lower tail of $X$ that in particular demonstrate that the bound of Theorem \ref{theorem1} is at least asymptotically tight.

\begin{theorem}\label{lemma2}
For any $t\in(0,1)$, we have that
\begin{equation}\label{X7}
\prob(X \leq t)=\prob(N>N^*),
\end{equation}
where $N\sim\text{Pois}(\lambda)$ and $N^*\sim\text{Pois}(-\log t)$ are independent.
\end{theorem}

\begin{proof}
Using Lemma \ref{lemma1}, we have for $t\in(0,1)$ that
\begin{align*}
\prob(X \leq t)&=\prob\left(\prod_{i=1}^NU_i \leq t, \ N>0\right)\\
&=\prob\left(-\sum_{i=1}^N\log U_i \geq -\log t, \ N>0\right).
\end{align*}

Now since $\{-\log U_i\}_{i\in\mathbb{N}}$ is an iid sequence of $\text{Exp}(1)$ random variables, we can associate with it a rate 1 Poisson process that is independent of $N$. Letting $N^*$ denote the number of arrivals by time $-\log t$ for this Poisson process, so that $N^*\sim\text{Pois}(-\log t)$, we see that
$$\left\{-\sum_{i=1}^n\log U_i \geq -\log t\right\}=\{n>N^*\}.$$

We can therefore conclude that
$$\prob(X \leq t)=\prob(N>N^*),$$
where $N$ and $N^*$ are independent.
\end{proof}

\begin{remark}\label{remark2}
In view of Theorem \ref{lemma2}, it is worth mentioning that the use of the union bound and concentration inequalities for the Poisson distribution fails to recover Theorem \ref{theorem1}. Nevertheless, Theorem \ref{lemma2} suggests that the bound of Theorem \ref{theorem1} can be generalized to one involving two independent Poisson random variables; this is done in the following corollary.
\end{remark}

\begin{corollary}\label{cor2}
Let $M_\mu$ and $M_\nu$ be two independent Poisson random variables with means $\mu$ and $\nu$, respectively. Then we have that
\begin{equation}\label{Pois}
\prob(M_\mu \geq M_\nu) \leq \exp\left\{-\left(\sqrt{\nu}-\sqrt{\mu}\right)_+^2\right\}.
\end{equation}
\end{corollary}
\begin{proof}
 Let $\beta>0$. Using Markov's Inequality, we observe that
\begin{align*}
\prob(M_\mu \geq M_\nu)&=\prob(e^{\beta M_\mu} \geq e^{\beta M_\nu})\\
&\leq\EX{e^{\beta M_\mu}}\EX{e^{-\beta M_\nu}}\\
&=\exp\left\{\mu(e^{\beta}-1)+\nu(e^{-\beta}-1)\right\}.
\end{align*}
Minimizing the above bound with respect to $\beta$ yields \eqref{Pois}.
\end{proof}

We are now ready to state and prove our asymptotic result. The proof makes use of a couple simple sets of inequalities which are given in the appendix.

\begin{theorem}\label{theorem2}
We have that
\begin{equation}\label{X8}
\log\prob(X \leq t)=-\left(\sqrt{-\log(t)}-\sqrt{\lambda}\right)^2+O\left(\log(-\log t)\right).
\end{equation}
\end{theorem}

\begin{proof}
By Theorem \ref{lemma2}, it suffices to consider $\prob(N>N^*)$, where $N\sim\text{Pois}(\lambda)$ and $N^*\sim\text{Pois}(-\log t)$ are independent. For a lower bound, we have that
\begin{align*}
\prob(N>N^*)&=\sum_{k=0}^{\infty}\prob(N^*=k)\prob(N>N^*|N^*=k)\\
&=\sum_{k=0}^{\infty}\frac{e^{-(-\log t)}(-\log t)^k}{k!}\prob(N \geq k+1)\\
&\geq t\sum_{k=0}^{\infty}\frac{(-\log t)^k}{k!}\frac{e^{-\lambda}\lambda^{k+1}}{(k+1)!}\hspace{70mm}\mbox{(Lemma \ref{lemma3})}\\
&=te^{-\lambda}\left(\sum_{k=1}^{\infty}\frac{(-\log t)^k\lambda^{k+1}}{k!(k+1)!}+\lambda\right)\\
&\geq te^{-\lambda}\left(\sum_{k=1}^{\infty}(-\log t)^k\lambda^{k+1}\frac{2^{2k}}{\sqrt{2\pi k}(2k+1)!}+\lambda\right)\hspace{30mm}\mbox{(Lemma \ref{lemma4})}\\
&\geq te^{-\lambda}\left(\frac{1}{\sqrt{2\pi}}\sum_{k=1}^{\infty}(-\log t)^k\lambda^{k+1}\frac{2^{2k}}{(2k+2)(2k+1)!}+\lambda\right)\\
&=te^{-\lambda}\left(-\frac{1}{4\sqrt{2\pi}\log t}\sum_{k=1}^{\infty}\frac{(2\sqrt{-\lambda\log t})^{2k+2}}{(2k+2)!}+\lambda\right)\\
&=te^{-\lambda}\left(-\frac{1}{4\sqrt{2\pi}\log t}\left(\frac{e^{2\sqrt{-\lambda\log t}}+e^{-2\sqrt{-\lambda\log t}}}{2}-1-\frac{(2\sqrt{-\lambda\log t})^2}{2}\right)+\lambda\right)\\
&\geq te^{-\lambda}\left(-\frac{1}{4\sqrt{2\pi}\log t}\left(\frac{e^{2\sqrt{-\lambda\log t}}}{2}-1+2\lambda\log t\right)+\lambda\right).
\end{align*}

We prove an analogous upper bound for $\lambda<2$ (see Remark \ref{remark3} following this proof for comments on how to extend it to the case $\lambda\geq2$). To this end, we observe that
\begin{align*}
\prob(N>N^*)&=\sum_{k=0}^{\infty}\frac{e^{-(-\log t)}(-\log t)^k}{k!}\prob(N \geq k+1)\\
&\leq t\sum_{k=0}^{\infty}\frac{(-\log t)^k}{k!}\frac{e^{-\lambda}\lambda^{k+1}}{(k+1)!}\frac{1}{1-\lambda/(k+2)}\qquad&&\mbox{(Lemma \ref{lemma3})}\\
&\leq\frac{te^{-\lambda}}{1-\lambda/2}\left(\sum_{k=1}^{\infty}\frac{(-\log t)^k\lambda^{k+1}}{k!(k+1)!}+\lambda\right)\\
&\leq\frac{te^{-\lambda}}{1-\lambda/2}\left(\sum_{k=1}^{\infty}(-\log t)^k\lambda^{k+1}\frac{(2\sqrt{2})2^{2k}}{\sqrt{2\pi k}(2k+1)!}+\lambda\right)\qquad&&\mbox{(Lemma \ref{lemma4})}\\
&\leq\frac{te^{-\lambda}}{1-\lambda/2}\left(\frac{1}{\sqrt{\pi}}\sum_{k=1}^{\infty}(-\log t)^k\lambda^{k+1}\frac{2^{2k+1}}{(2k+1)!}+\lambda\right)\\
&=\frac{te^{-\lambda}}{1-\lambda/2}\left(\sqrt{\frac{\lambda}{-\pi\log t}}\sum_{k=1}^{\infty}\frac{(2\sqrt{-\lambda\log t})^{2k+1}}{(2k+1)!}+\lambda\right)\\
&\leq\frac{te^{-\lambda}}{1-\lambda/2}\left(\sqrt{\frac{\lambda}{-\pi\log t}}e^{2\sqrt{-\lambda\log t}}+\lambda\right).
\end{align*}
\end{proof}

\begin{remark}\label{remark3}
Note that our use of the upper bound provided by Lemma \ref{lemma3} in the proof is only valid for $\lambda<2$, and it is slightly messier (though straightforward) to handle the case $\lambda\geq2$. For this case, we simply observe that
\begin{align*}
\prob(N>N^*)&=\sum_{k=0}^{\infty}\frac{e^{-(-\log t)}(-\log t)^k}{k!}\prob(N \geq k+1)\\
&=t\sum_{k=0}^{\lfloor\lambda\rfloor}\frac{(-\log t)^k}{k!}\prob(N \geq k+1)+t\sum_{k=\lfloor\lambda\rfloor+1}^{\infty}\frac{(-\log t)^k}{k!}\prob(N \geq k+1)\\
&\leq t\sum_{k=0}^{\lfloor\lambda\rfloor}\frac{(-\log t)^k}{k!}+t\sum_{k=\lfloor\lambda\rfloor+1}^{\infty}\frac{(-\log t)^k}{k!}\prob(N \geq k+1).
\end{align*}

The first term above is $O(t(-\log t)^{\lfloor\lambda\rfloor})$, while the second term can be handled exactly as in the proof of Theorem \ref{theorem2} (i.e., by applying the upper bound in Lemma \ref{lemma3}) since $k+2>\lambda$ for all $k\geq\lfloor\lambda\rfloor+1$. The result of the theorem remains unchanged.
\end{remark}

\bibliographystyle{Chicago}
\bibliography{references}

\section{Appendix}\label{appendix}

We borrow from folklore the following simple set of inequalities for the upper tail of a Poisson random variable.

\begin{lemma}\label{lemma3}
Let $Y\sim\text{Pois}(\mu)$. Then for all $\mu<n+1$, we have that
\begin{equation}\label{pois}
\frac{e^{-\mu}\mu^n}{n!}\leq\prob(Y \geq n)\leq\frac{e^{-\mu}\mu^n}{n!}\frac{1}{1-\mu/(n+1)}.
\end{equation}
\end{lemma}

\begin{proof}
The lower bound is trivial. For the upper bound, we observe for $\mu<n+1$ that
$$\prob(Y \geq n)=\sum_{k=n}^{\infty}\frac{e^{-\mu}\mu^k}{k!}\leq\frac{e^{-\mu}\mu^n}{n!}\sum_{k=0}^{\infty}\frac{\mu^k}{(n+1)^k}=\frac{e^{-\mu}\mu^n}{n!}\frac{1}{1-\mu/(n+1)}.$$
\end{proof}

The next lemma is an easy consequence of Stirling's approximation.

\begin{lemma}\label{lemma4}
For all $n\in\mathbb{N}$, we have that
\begin{equation}\label{stirling}
\frac{2^{2n}}{\sqrt{2\pi n}(2n+1)!}\leq\frac{1}{n!(n+1)!}\leq2\sqrt{2}\frac{2^{2n}}{\sqrt{2\pi n}(2n+1)!}.
\end{equation}
\end{lemma}

\begin{proof}
Let $n\in\mathbb{N}$. We begin by observing that the inequalities
$$\frac{(n!)^2}{2(2n)!}\leq\frac{n!(n+1)!}{(2n+1)!}\leq\frac{(n!)^2}{(2n)!}$$
hold, which combined with Stirling's bounds (see, e.g., \cite{feller1968}),
$$e^{\frac{1}{12n+1}}\sqrt{2\pi}n^{n+1/2}e^{-n} \leq n! \leq e^{\frac{1}{12n}}\sqrt{2\pi}n^{n+1/2}e^{-n},$$
yield that
$$\frac{e^{\frac{2}{12n+1}-\frac{1}{24n}}\sqrt{2\pi n}}{2^{2n+3/2}}\leq\frac{n!(n+1)!}{(2n+1)!}\leq\frac{e^{\frac{1}{6n}-\frac{1}{24n+1}}\sqrt{2\pi n}}{2^{2n+1/2}}.$$
The preceding inequalities simplify to the inequalities
$$\frac{1}{2\sqrt{2}}\frac{\sqrt{2\pi n}}{2^{2n}}\leq\frac{n!(n+1)!}{(2n+1)!}\leq\frac{\sqrt{2\pi n}}{2^{2n}},$$
which are an easy rearrangement of those in \eqref{stirling}.
\end{proof}

\end{document}